\documentclass[10pt,a4paper,oneside]{amsproc}

\usepackage{latexsym, amsthm}
\usepackage{amsfonts, amsmath, amssymb}
\usepackage{euscript,mathrsfs}
\usepackage{url}
\usepackage{array}
\usepackage[a4paper]{geometry}
\usepackage{graphics}
\usepackage[usenames]{color}
\usepackage[all]{xy}
\usepackage{graphicx}
\usepackage{boxedminipage}

\newtheorem{theorem}{Theorem}[section]

\newtheorem{conjecture}{Conjecture}
\newtheorem{corollary}[theorem]{Corollary}
\newtheorem{question}[theorem]{Question}
\newtheorem{definition}[theorem]{Definition}
\newtheorem{remark}[theorem]{Remark}

\newtheorem{lemma}[theorem]{Lemma}

\newtheorem{proposition}[theorem]{Proposition}

\newcommand{\codim}{\mathrm{codim}}

\newcommand{\Fqt}{\mathbb{F}_{q^2}}
\newcommand{\Fq}{\mathbb{F}_q}

\def\Fq{{\mathbb F}_q}
\def\AA{{\mathbb A}}
\def\FF{{\mathbb F}}
\def\PP{{\mathbb P}}
\newcommand{\V}{{\mathsf{V}}}

\begin{document}

\title[Intersection of Hermitian surfaces and cubic surfaces]{Maximum number of points on intersection of a cubic surface and a non-degenerate Hermitian surface}
\author{Peter Beelen}
\address{Department of Applied Mathematics and Computer Science, \newline \indent
Technical University of Denmark, DK 2800, Kgs. Lyngby, Denmark}
\email{pabe@dtu.dk}

\author{Mrinmoy Datta}
\address{Department of Applied Mathematics and Computer Science, \newline \indent
Technical University of Denmark, DK 2800, Kgs. Lyngby, Denmark}
\email{mrinmoy.dat@gmail.com}
\thanks{The second named author is supported by The Danish Council for Independent Research (Grant No. DFF--6108-00362).}


\begin{abstract}
In 1991 S{\o}rensen proposed a conjecture for the maximum number of points on the intersection of a surface of degree $d$ and a non-degenerate Hermitian surface in $\PP^3(\Fqt)$. The conjecture was proven to be true by Edoukou in the case when $d=2$. In this paper, we prove that the conjecture is true for $d=3$. For $q \ge 4$, we also determine the second highest number of rational points on the intersection of a cubic surface and a non-degenerate Hermitian surface. Finally, we classify all the cubic surfaces that admit the highest and, for $q \ge 4$, the second highest number of points in common with a non-degenerate Hermitian surface. This classification disproves a conjecture proposed by Edoukou, Ling and Xing.
\end{abstract}

\date{}
\maketitle

\section{Introduction}
Hermitian varieties defined over a finite field have received a lot of attention in the literature since they were introduced \cite{BC} in 1966 by Bose and Chakravarti. The geometry of the Hermitian varieties has been studied extensively in \cite{BC} and further in \cite{C}. In particular, the line-plane incidence with respect to the non-degenerate Hermitian surfaces gives rise to beautiful combinatorial structures. Various combinatorial studies related to Hermitian surfaces include \cite{CK, GK, IZ2} among others. Further, Hermitian varieties have turned out to be very efficient from the perspective of error correcting codes since they have a large number of rational points. For discussions on codes defined by homogeneous polynomials of a fixed degree on Hermitian surfaces, we refer to \cite[Example 6.6]{L} and more generally on Hermitian varieties of arbitrary dimension in \cite[Section 3]{E}. In order to determine the minimum distance of the codes on Hermitian surfaces mentioned above, as well as from independent interests one may ask the following question:

\begin{question}\label{q} \normalfont
Let $F \in \Fqt[x_0, x_1, x_2, x_3]$ be a homogeneous polynomial of degree $d$ and $V_2$ denote a non-degenerate Hermitian surface in $\PP^3 (\Fqt)$. What is the maximum number of $\Fqt$- rational points in $V (F) \cap V_2$?
\end{question}

Note that answering this question will lead to determining the maximum number of rational points on a hyperplane section of the $d$-uple embedding of the Hermitian surface. It appears that this question was first addressed by S{\o}rensen in his Ph.D. thesis \cite{SoT} in order to generalize the work of Chakravarti \cite{Ch} towards understanding the $2$-uple embedding of the cubic surface defined by the equation $x_0^3 + x_1^3 + x_2^3 + x_3^3 = 0$ in $\PP^3(\mathbb{F}_4)$. S{\o}rensen made the following conjecture:

\begin{conjecture}\cite[Page 9]{SoT}\label{So}
Let $F$ and $V_2$ be as in Question \ref{q}. If $d \le q$, then
$$|V(F)(\Fqt) \cap V_2| \le d (q^3 + q^2 - q) + q + 1.$$
Further, the surfaces given by a homogeneous polynomial  $F \in \Fqt[x_0, x_1, x_2, x_3]$ attaining the above upper bound are given by a union of $d$ planes in $\PP^3 (\Fqt)$ that are tangent to $V_2$, each containing a common line $\ell$ intersecting $V_2$ at $q+1$ points.
\end{conjecture}

In \cite[Sections 7 and 8]{BC} Bose and Chakravarti have analysed the linear sections of general Hermitian varieties and determined the number of $\Fqt$-rational points they contain. In particular, this answers the Question \ref{q} for $d=1$. Moreover, it can easily be seen that their answer validates Conjecture \ref{So} in this case. S{\o}rensen further observed that the combinatorial complexity of finding the maximum number of points of intersection of the Hermitian surface and a surface of degree $d$ increases significantly with $d$. Conjecture \ref{So} stands open till date.

The first breakthrough towards proving the conjecture was made by Edoukou in 2007. In \cite{E}, he proved that the conjecture is true for $d=2$. Subsequently, in \cite{ELX} Edoukou, Ling and Xing determined the first five highest number of points that an intersection of a quadric surface and a non-degenerate Hermitian surface can have in $\PP^3 (\Fqt)$. In the same article, the authors made several conjectures related to the configuration of hypersurfaces which may admit several highest numbers of points of intersection with non-degenerate Hermitian surfaces.

In the current paper, we answer Question \ref{q} for $d=3$. Certainly, cubic surfaces have been one among the most fascinating and studied objects in algebraic geometry and in particular, some of the results in Chapter 7 of \cite{R} have turned out be extremely useful for us. In Theorem \ref{mt} and Corollary \ref{class} we give a proof of Conjecture \ref{So} for $d=3$. To prove Theorem \ref{mt} we make extensive use of the underlying combinatorial structure of line-plane incidence with respect to Hermitian surfaces. Further, for $q \ge 4$ we classify all the cubic surfaces that attain the second highest number of points of intersection of a cubic surface and a non-degenerate Hermitian surface. This, in particular, disproves one of the conjectures in \cite{ELX}.


This paper is organized as follows. In Section \ref{sec:prel}, we recall various well-known properties of non-degenerate Hermitian surfaces, revisit some preliminary results from algebraic geometry and some basic bounds on the number of rational points on varieties defined over a finite field.
Section \ref{sec:cubic} deals with results on cubic surfaces. In particular, we analyze cubic surfaces in $\PP^3(\Fqt)$ containing of a pair of skew lines. In Section \ref{sec:general} we derive various results that are helpful in partially answering Question \ref{q} for general values of $d$. Finally, in Section \ref{sec:mt} we present our main results for $d =3$.


\section{Preliminaries}\label{sec:prel}
Fix a prime power $q$ for the remainder of this paper. As usual, $\Fq$ and $\Fqt$ denote the finite fields with $q$ and $q^2$ elements respectively. For $m \ge 0$, we denote by $\PP^m$, the projective space of dimension $m$ over the algebraic closure $\overline{\FF}_q$, while $\PP^m (\Fqt)$ will denote the set of all $\Fqt$-rational points of $\PP^m$.
Similarly, $\AA^m$ and $\AA^m (\Fqt)$ will denote the affine space of dimension $m$ over $\overline{\FF}_q$ and $\Fqt$ respectively.
Further, for a homogeneous polynomial $F \in \Fqt[x_0, \dots , x_m]$, we denote by $V(F)$, the set of zeroes of $F$ in $\PP^m$ and by $\V(F)=V(F)(\Fqt)$ the set of all $\Fqt$-rational points of $V(F)$. By an algebraic variety we will mean a set of zeroes of a certain set of polynomials in the affine space or projective space, depending on the context. In particular, an algebraic variety need not be irreducible. We remark that, whenever we say that a variety is irreducible or nonsingular, we will mean that the variety is irreducible or nonsingular over $\overline{\FF}_q$. This section is divided into three subsections: in the first subsection, we recall several known facts about Hermitian varieties over finite fields, the second subsection is dedicated to some useful results from basic algebraic geometry, while the third subsection concerns some known upper bounds on the number of rational points on varieties defined over a finite field. The proofs of the results in this section can be found in the indicated references.

\subsection{Hermitian varieties over finite fields}\label{sec:herm}

In this subsection, we recall the definition of Hermitian varieties and  various of their well-known properties (cf. \cite{BC, C}) that will be used in the latter part of this paper.  At the end of this subsection, we recall the result proved by Edoukou \cite{E} where he proves Conjecture \ref{So} for the case $d=2$. We begin with the following.

\begin{definition} \normalfont
For an $(m+1) \times (m+1)$ matrix $A = (a_{ij})$ ($0 \le i, j \le m$) with entries in $\Fqt$, we denote by $A^{(q)}$, the matrix whose $(i, j)$-th entry is given by ${a_{ij}}^q$. The matrix $A$ is said to be a \textit{Hermitian matrix} if $A \neq 0$ and $A^T = A^{(q)}$.

A \textit{Hermitian variety} of dimension $m-1$, denoted by $V_{m-1}$, is the set of zeroes of the polynomial $x^T A x^{(q)}$ inside $\PP^m$, where $A$ is an $(m+1) \times (m+1)$ Hermitian matrix and $x = (x_0, \dots, x_m)^T$. The Hermitian variety is said to be \textit{non-degenerate} if $\mathrm{rank} \ A = m+1$ and \textit{degenerate} otherwise.
\end{definition}

It is a well-known fact that if the rank of a Hermitian matrix is $r$, then by a suitable change of coordinates, we can describe the corresponding Hermitian variety by the zero set of the polynomial
\begin{equation}\label{herm}
x_0^{q+1} + x_1^{q+1} + \dots + x_{r-1}^{q+1} = 0.
\end{equation}
For a proof of the above fact the reader is referred to \cite[Equation (5.6)]{BC}. We note that the polynomial $x_0^{q+1} + \dots + x_{r-1}^{q+1}$ is irreducible over the algebraic closure of $\Fq$ whenever $r \ge 3$.
This shows that Hermitian varieties corresponding to Hermitian matrices of rank at least $3$ are irreducible. For the purpose of this paper we will, from now on, restrict our attention to Hermitian curves and Hermitian surfaces, i.e. Hermitian varieties of dimensions $1$ and $2$ respectively. The linear sections of Hermitian surfaces are extremely well understood. Indeed, the hyperplane section of a Hermitian variety is also a Hermitian variety. We recall the following two results of Bose and Chakravarti \cite{BC} concerning the intersections of lines and planes with Hermitian surfaces in $\PP^3 (\Fqt)$.

\begin{lemma}\cite[Section 7]{BC}\label{line}
Any line in $\PP^3 (\Fqt)$ satisfies precisely one of the following.
\begin{enumerate}
\item[(i)] The line intersects $V_2$ at exactly $1$ point.
\item[(ii)] The line intersects $V_2$ at exactly $q+1$ points.
\item[(iii)] The line is contained in $V_2$.
\end{enumerate}
\end{lemma}
\noindent Reflecting these three possibilities, we give the following definition.
\begin{definition}\label{lines}\normalfont
Let $\ell$ be a line in $\PP^3(\Fqt)$.  The line $\ell$ is called
\begin{enumerate}
\item[(a)]  a \textit{tangent line} if it intersects $V_2$ at exactly $1$ point.
\item[(b)]  a \textit{secant line} if it intersects $V_2$ at exactly $q+1$ points.
\item[(c)] a \textit{generator} if it is contained in $V_2$.
\end{enumerate}
\end{definition}

\begin{theorem}\cite[Section 10]{BC}\label{linear}
Let $V_2$ denote a non-degenerate Hermitian surface in $\PP^3$. Let $\Pi$ be any hyperplane in $\PP^3 (\Fqt)$. If $\Pi$ is a tangent to $V_2$ at some point $P \in V_2$, then $\Pi$ intersects $V_2$ at exactly $q+1$ generators, all passing through $P$. Otherwise, $\Pi$ intersects $V_2$ at a non-degenerate Hermitian curve $V_1$. In particular,
\begin{equation*}
|V_2 \cap \Pi| =
 \begin{cases}
q^3 + q^2 + 1 \ \ \ \ \mathrm{if} \ \Pi \ \mathrm{is \ a \ tangent \ plane,} \\
q^3 + 1 \ \ \ \ \ \ \ \ \ \ \mathrm{if} \ \Pi \ \mathrm{is \ not \ a \ tangent \  plane}.
\end{cases}
\end{equation*}
\end{theorem}

\begin{remark}\label{linetp} \normalfont
Let $\Pi$ be tangent to $V_2$ at a point $P$. Theorem \ref{linear} shows that there are $q+1$ lines passing through $P$ contained in $\Pi$ that are generators. The remaining $q^2 - q$ lines contained in $\Pi$ that pass through $P$ are tangent lines. Further, any line that passes through $P$ but is not contained in $\Pi$ is a secant line. We refer to \cite[Section 10]{BC} for the proof of these results.
\end{remark}

In the course of proving our main results in Section \ref{sec:mt}, we shall make frequent use of the set of all planes containing a given line. Let $\ell$ be any line in $\PP^3 (\Fqt)$. By \textit{the book of planes around} $\ell$, denoted by $\mathcal{B}(\ell)$, we mean the set of all planes in $\PP^3(\Fqt)$ that contain $\ell$. This set of planes is also called the pencil or sheaf of planes with axis $\ell$. We note that, for any line $\ell$ in $\PP^3(\Fqt)$, the corresponding book has cardinality $q^2 + 1$. The following proposition and its corollary will prove to be instrumental in the latter part of this article.

\begin{proposition}\label{descbook}
Let $\ell$ be a line in $\PP^3(\Fqt)$ and $\mathcal{B}(\ell)$ be the book of planes around $\ell$.
\begin{enumerate}
\item[(a)] \cite[Lemma 5.2.3]{C} If $\ell$ is a generator, then every plane in $\mathcal{B}(\ell)$ is tangent to $V_2$ at some point of $\ell$.
\item[(b)] \cite[Lemma 5.2.6]{C} If $\ell$ is a tangent line, then exactly one plane in $\mathcal{B}(\ell)$ is tangent to $V_2$ at the point where $\ell$ meets $V_2$.
\item[(c)] \cite[Lemma 5.2.5]{C} If $\ell$ is a secant line, intersecting $V_2$ at $q+1$ points $P_0, \dots, P_q$, then out of the $q^2 + 1$ planes in $\mathcal{B}(\ell)$, exactly $q+1$ are tangent to $V_2$ at some point distinct from $P_0, \dots, P_q$.
\end{enumerate}
\end{proposition}

\begin{corollary}\label{int}
Let $\Pi_1$ and $\Pi_2$ be two planes in $\PP^3 (\Fqt)$ meeting at a line $\ell$. Then
\begin{itemize}
\item[(a)] If $\Pi_1, \Pi_2$ are both tangent to $V_2$, then $\ell$ is either a generator or a secant line.
\item[(b)] If one of $\Pi_1, \Pi_2$ is not a tangent plane, then $\ell$ is either a secant or a tangent line.
\end{itemize}
\end{corollary}

\begin{proof}
Part (a) follows trivially by noticing that  the book of a tangent line contains exactly one tangent plane (see Prop. \ref{descbook} (b)), while part (b) is an easy consequence of the fact that all the planes in the book of a generator are tangent planes (see Prop. \ref{descbook} (a)).
\end{proof}

Based on the above combinatorial structure, S{\o}rensen considered \cite{SoT} the following arrangement of planes that attains the conjectured upper bound. We include a proof for the convenience of the reader.

\begin{proposition}\label{attained}
Let $\Pi_1, \dots, \Pi_d$ be $d$ distinct planes that are tangent to $V_2$. Further assume that they contain a common line which is a secant. Then,
$$|(\Pi_1 \cup \dots \cup \Pi_d) \cap V_2| = d(q^3 + q^2 - q) + q + 1.$$
Moreover, there exists a homogeneous polynomial $F \in \Fqt[x_0, x_1, x_2, x_3]$ of degree $d$  such that $V(F) = \Pi_1 \cup \dots \cup \Pi_d$, where the $\Pi_i$-s are defined as above.
\end{proposition}

\begin{proof}
Let $\ell = \Pi_1 \cap \dots \cap \Pi_d$. By hypothesis the line $\ell$ is a secant to $V_2$. Now for $k =1, \dots, d$ we have $|(\Pi_k \setminus \ell) \cap V_2| = q^3 + q^2 + 1 - (q+1) = q^3 + q^2 - q$. Consequently,
$$|(\Pi_1 \cup \dots \cup \Pi_d) \cap V_2| = \sum_{k=1}^d |(\Pi_k \setminus \ell) \cap V_2|+ |\ell \cap V_2| = d(q^3 + q^2 - q) + q+1.$$
Since $d \le q$, the existence of distinct planes $\Pi_1, \dots, \Pi_d$ tangent to $V_2$ each containing a common secant line is guaranteed by Proposition \ref{descbook} (c). Existence of $F \in \Fqt[x_0, x_1, x_2, x_3]$ homogeneous of degree $d$ such that $V(F) = \Pi_1 \cup \dots \cup \Pi_d$ follows trivially.
\end{proof}

\noindent Finally, to conclude this subsection, we recall the following results of Edoukou.

\begin{theorem}\cite[Thm. 5.11 and 6.2]{E}\label{Ed}
Let $F \in \Fqt[x_0, x_1, x_2, x_3]$ be a homogeneous quadratic polynomial. Then,
$$|\V(F) \cap V_2| = 2 q^3 + 2 q^2 - q + 1 \ \ \ \ \ \mathrm{or} \ \ \ \ \ \ \ |\V(F) \cap V_2| \le 2q^3 + q^2 + 1.$$
Also, $|\V(F) \cap V_2| = 2 q^3 + 2 q^2 - q + 1$ if and only if $V(F)$ is the union of two tangent planes intersecting at a secant line.
\end{theorem}

\subsection{Preliminaries from algebraic geometry}
In the present subsection we recall various basic results from algebraic geometry that will be needed for proving our main theorem. We will make use of the notions of dimension, degree and singularity of a variety, as can be found in standard textbooks of Algebraic Geometry, for example, the book of Harris \cite{H}. We begin with the following definitions. A variety is said to be of \textit{pure dimension} or \textit{equidimensional} if all the irreducible components of the variety have equal dimension. Further, two equidimensional varieties $X, Y \subset \PP^m$ are said to \textit{intersect properly} if $\codim (X \cap Y) = \codim \ X + \codim \  Y.$ It turns out that good upper bounds for the number of rational points on varieties defined over a finite field, depend on the degree and dimension of the variety. Because of this reason, the following proposition from \cite{H} will be indispensable for us.

\begin{proposition}\cite[Cor. 18.5 and 18.6]{H}\label{deg}
Let $X$ and $Y$ be equidimensional varieties that intersect properly in $\PP^m$. Then
$\deg (X \cap Y) \le \deg X \deg Y$ and
equality holds only if both $X$ and $Y$ are smooth at a general point of any component of $X \cap Y$.
\end{proposition}

\noindent As can be seen in \cite[p. 54]{H}, the phrase ``$X$ is smooth at a general point of any component of $X \cap Y$" means the following: for any irreducible component $C$  of $X \cap Y$,  the set of points on $C$ where $X$ is nonsingular, contains an open dense subset of $C$. In particular, we have the following proposition concerning the intersection of two hypersurfaces.
\begin{proposition}\label{coprime}
Let $F, G \in \Fqt[x_0, x_1, \dots, x_m]$ be nonconstant homogeneous polynomials having no common factors. Then
\begin{enumerate}
\item[(a)] $V(F)$ and $V(G)$ intersect properly.
\item[(b)] $V(F) \cap V(G)$ is equidimensional of dimension $m-2$.
\item[(c)] $\deg (V(F) \cap V(G)) \le \deg F \deg G$. Moreover, if there exists an irreducible component $C$ of $V(F) \cap V(G)$ such that $V(G)$ is singular at every point of $C$, then $$\deg (V(F) \cap V(G)) \le \deg F \deg G - 1.$$
\end{enumerate}
\end{proposition}

\begin{proof}
Part (a) can be found in \cite[Lemma 2.2]{DG}. Part (b) is proved using Macaulay's unmixedness theorem (See \cite[Chapter 7, Theorem 26]{ZS}). Part (c) follows straightaway from Proposition \ref{deg}.
\end{proof}

\begin{remark}
\normalfont
As an immediate consequence of the irreducibility of non-degenerate Hermitian varieties of dimension at least $1$ and Proposition \ref{coprime}, we see that a surface given by a nonconstant homogeneous polynomial of degree $d \le q$ in $\Fqt[x_0, x_1, x_2, x_3]$ intersects the Hermitian surface at an equidimensional variety of dimension $1$ and degree at most $d(q+1).$ A similar consequence for non-degenerate Hermitian curves can be derived.
\end{remark}

\subsection{Basic upper bounds}
In this subsection, we recall some well-known upper bounds on the  number of rational points on varieties defined over a finite field $\Fq$ with given degree and dimension. We start with a result from \cite{LR}.

\begin{proposition}\cite[Prop. 2.3]{LR}\label{lac}
Let $X$ be an equidimensional projective (resp. affine) variety defined over a finite field $\Fq$. Further assume that $\dim X = \delta$ and $\deg X = d$.  Then
$$|X(\Fq)| \le d p_{\delta} \ \ \ \ \ \ (\mathrm{resp.} \ \  |X (\Fq)| \le d q^{\delta}),$$
where $p_{\delta} = 1 + q + \dots + q^{\delta}$.
\end{proposition}

The projective part of the above proposition appears in \cite[Proposition 2.3]{L} and in \cite[Proposition 12.1]{GL} in somewhat incorrect form. In \cite{LR}, the authors observed that the condition of equidimensionality has to be added in the hypothesis to make it correct; in Proposition 2.3 of the same article, the upper bound is proved for irreducible varieties and the statement above follows immediately. An alternative proof can be found in \cite[Proposition 2.3]{DG1}.

The following theorem was proved by Serre \cite{S} and independently by S{\o}rensen \cite{So}. It concerns the maximum number of zeroes a homogeneous polynomial of degree $d$  in $m+1$ variables can have in $\PP^m (\Fq)$. The upper bound is often referred to as \textit{Serre's inequality} in the literature.

\begin{theorem}[Serre's inequality]\label{serre}
Let $ F \in \Fq[x_0, x_1, \dots, x_m]$ be a nonconstant homogeneous polynomial of degree $d \le q$. Then
$$|\V(F)| \le dq^{m-1} + p_{m-2},$$
where $p_{m-2} = 1 + q + \dots + q^{m-2}$. Moreover,
equality holds if and only if $V(F)$ is a union of $d$ hyperplanes defined over $\Fq$ all containing a common linear subspace of codimension $2$.
\end{theorem}

\section{Cubic surfaces}\label{sec:cubic}
For future use, we derive some results on cubic surfaces $X \subset \mathbb{P}^3$ defined over a field $\mathbb{F}$ and containing a skew pair of lines $\ell$ and $m$. In later sections, we will apply these results for the field $\mathbb{F}=\Fqt.$ We may assume after a linear change of coordinates, if necessary, that $\ell = V(x_2, x_3)$ and $m=V(x_0,x_1).$ Such a cubic surface is the zero set of a homogeneous polynomial of the form
\begin{equation}
\label{cubic1}
F = Ax_0^2 + B x_0 x_1 + C x_1^2 + D x_0 + E x_1,
\end{equation}
where $A, B, C, D, E \in \mathbb{F} [x_2, x_3]$, with $\deg A = \deg B = \deg C = 1$ and  $\deg D = \deg E = 2$.
In this section, we are interested in estimating the cardinality of the set of planes
$$\mathcal{T}_{\ell} := \{\Pi \in \mathcal{B}(\ell) \mid V(F) \cap \Pi \ \mathrm{is \ a \ union \ of  \ lines}\}.$$
The reason for studying such sets of planes will become clear in later sections.

To study the behavior of the cubic surface $X=V(F)$ restricted to a plane $\Pi \in \mathcal{B}(\ell)$, we start with the following lemma, which essentially is a direct consequence of the proof of Proposition 7.3 from \cite{R}. However, since in \cite{R} it was assumed that the cubic is smooth and the proof was given for the case that $\mathrm{char}(\mathbb{F})$ is odd, we provide a proof for the convenience of the reader.
\begin{lemma}\label{lem:red}
Let $\mathbb{F}$ be an algebraically closed field and $a,b,c,d,e \in \mathbb{F}$ not all zero. Further, let $g:=a x_0^2 + bx_0x_1 + c x_1^2 + dx_0t + e x_1t \in \mathbb{F}[x_0,x_1,t]$. Then $g$ is reducible if and only if $-ae^2+bde-cd^2=0.$
\end{lemma}
\begin{proof}
Since the polynomial $g$ has degree two, it is reducible if and only if there exists a projective point which is a common zero of $g$ and its partial derivatives. 
Considering the partial derivatives of $g$, we see that the coordinates of such a point $P=[x_0:x_1:t]$ need to satisfy
$$\begin{pmatrix}
2a && b && d \\
b && 2c && e \\
d && e && 0
\end{pmatrix}
\begin{pmatrix}
x_0 \\
x_1 \\
t
\end{pmatrix}
=\begin{pmatrix}
0 \\
0 \\
0
\end{pmatrix}.
$$
Let us denote the $3 \times 3$ matrix occurring above by $M$. Observe that $\det M=2(-ae^2+bde-cd^2).$

Now assume that $g$ and its partial derivatives have some projective point as common zero. Then necessarily $\det M=0.$ If $\mathrm{char}(\mathbb{F}) \neq 2$, then $-ae^2+bde-cd^2=0$, as desired. If $\mathrm{char}(\mathbb{F})=2$, it is easy to see that the rank of $M$ equals two unless $(e,d,b)=(0,0,0).$ If $(e,d,b)=(0,0,0),$ clearly $-ae^2+bde-cd^2=0$, while otherwise the only common zero of the partial derivatives is the projective point $[e:d:b].$ By assumption, this point also needs to be a zero of $g$, which implies that $-ae^2+bde-cd^2=0$.

Now conversely assume that $-ae^2+bde-cd^2=0$. As in the proof of Proposition 7.3 in \cite{R}, if $\mathrm{char}(\mathbb{F}) \neq 2$, then $g$ defines a singular conic and is hence reducible. If $\mathrm{char}(\mathbb{F})=2$ and $(e,d,b)=(0,0,0)$, then $g$ is a square. This leaves the case that $\mathrm{char}(\mathbb{F})=2$ and $(e,d,b) \neq (0,0,0).$ In this case, the projective point $[e:d:b]$ is a singularity on $V(g)$, again implying that $g$ is reducible, since $\deg g=2$.
\end{proof}

We now obtain the following proposition on the cardinality of $\mathcal{T}_\ell$. If $X$ is smooth, this proposition is a direct consequence of Proposition 7.3 in \cite{R}, but the proof given there applies for any cubic surface. For the convenience of the reader, we include the proof.

\begin{proposition}\label{cubenz}
Let $X=V(F)$ be a cubic surface, with $F$ as in equation \eqref{cubic1}. Define $\ell=V(x_2,x_3)\subset X$ and $\mathcal{T}_{\ell} = \{\Pi \in \mathcal{B}(\ell) \mid V(F) \cap \Pi \ \mathrm{is \ a \ union \ of  \ lines}\}.$ Then $\mathcal{T}_\ell=\mathcal{B}(\ell)$ or $|\mathcal{T}_{\ell}| \le 5$.
\end{proposition}
\begin{proof}
It is enough to show the lemma in case $\mathbb{F}$ is algebraically closed. Indeed, if the proposition is true under this assumption, it will be true for any field. Any $\Pi \in \mathcal{B}(\ell)$ is of the form $\Pi =V(\lambda_1 x_2 -  \lambda_2 x_3)$, where $[\lambda_1: \lambda_2] \in \PP^1(\mathbb{F})$. Now, the last two coordinates of any point on $\Pi$ are given by $[\lambda_2 : \lambda_1]$. Here, without loss of generality we may assume that the first nonzero coordinate of $[\lambda_2: \lambda_1]$ is $1$. We have
$$F|_{(\Pi \setminus \ell)} = a x_0^2 + bx_0x_1 + c x_1^2 + dx_0t + e x_1t,$$
where $a = A(\lambda_2, \lambda_1), b = B(\lambda_2, \lambda_1 ), c = C(\lambda_2, \lambda_1), d = D(\lambda_2, \lambda_1), e = E(\lambda_2, \lambda_1)$, and $t$ is $x_2$ if $\lambda_1 \neq 0$ and $x_3$ otherwise.
Now apart from $\ell$,  $V(F|_{\Pi})$ contains  additional lines if and only if $g (x_0, x_1, t) = a x_0^2 + bx_0x_1 + c x_1^2 + dx_0 t+ e x_1t$ is reducible.
However, Lemma \ref{lem:red} implies that $g$ is reducible if and only if $-ae^2+bde-cd^2=0.$ 

This implies that, for any $\Pi \in \mathcal{B}(\ell)$, the cubic curve $V(F) \cap \Pi$ is a union of lines if and only if $-ae^2 + bde - cd^2=0$. Note that, either $-AE^2 + BDE - CD^2$ is the zero polynomial or it is a nonzero quintic. In the first case, we have $\mathcal{T}_{\ell} = \mathcal{B}(\ell)$, while in the second case $|\mathcal{T}_{\ell}| \le 5$.
\end{proof}

Note that if $|\mathbb{F}| \ge 5$, then the cases $\mathcal{T}_\ell=\mathcal{B}(\ell)$ and $|\mathcal{T}_{\ell}| \le 5$ are disjoint. Our next result concerns a description of cubic surfaces defined by a polynomial as in equation \eqref{cubic1} satisfying $-AE^2+BDE-CD^2=0.$ We start with a definition and a lemma.
\begin{definition} \normalfont
Let $X \subset \mathbb{P}^3$ be a cubic surface defined over a field $\mathbb{F}.$ A point $P$ on $X$ is called a weak Eckardt point, if it is a point of intersection of three lines on $X$. If the point $P$ is a smooth point on $X$, it is called an Eckardt point.
\end{definition}

Note that since an Eckardt point $P$ is assumed to be a smooth point on $X$, there are exactly three lines contained in $X$ through $P$, all contained in the tangent plane of $X$ at $P$. By definition, a singular point $P$ of $X$ cannot be an Eckardt point, though it may be a weak Eckardt point. In this case more than three lines contained in $X$ may pass through $P$. One possibility is that $X$ is a cone with center $P$, but also if $X$ is not a cone, this may happen. Consider for example $X=V(F)$, with $F=x_0x_1(x_2-ax_3)-(x_1-x_2)(x_1-x_3)(x_2-x_3)$ and where $a \in \mathbb{F} \setminus \{0,1\}.$ Then $P=[1:0:0:0]$ is a weak Eckardt point lying on exactly six distinct lines each contained in $X$.

We now show that if a line on $X$ contains an Eckardt point, then any other weak Eckardt point on that line is smooth and hence an Eckardt point.

\begin{lemma}\label{weakEckardt}
Let $X \subset \mathbb{P}^3$ be an irreducible cubic surface defined over a field $\mathbb{F}$ and $\ell$, $m$ two skew lines on $X$. Assume that $\ell$ is not a double line. If $\ell$ contains an Eckardt point, then any point on $\ell$ is smooth. In particular, any other weak Eckardt point of $X$ on $\ell$ is actually an Eckardt point.
\end{lemma}
\begin{proof}
Let $P_1,P_2 \in \ell$ and $P_1$ be an Eckardt point. Choosing a suitable coordinate system $x_0$, $x_1$, $x_2$, $x_3$ for $\mathbb{P}^3(\mathbb{F}),$ we may assume that $\ell=V(x_0,x_1)$, $m=V(x_2,x_3)$,  $P_1=[0:0:1:0]$ and $P_2=[0:0:0:1].$ This means that $X=V(F)$, with $F$ of the form
\begin{multline*}
F
=
a_2x_0^2x_2+a_3x_0^2x_3+b_2x_0x_1x_2+b_3x_0x_1x_3+c_2x_1^2x_2+c_3x_1^2x_3\\
+d_{22}x_0x_2^2+d_{23}x_0x_2x_3+d_{33}x_0x_3^2
+e_{22}x_1x_2^2+e_{23}x_1x_2x_3+e_{33}x_1x_3^2.
\end{multline*}
Since $P_1$ is a smooth point of $X$, we have $(d_{22},e_{22}) \neq (0,0)$ and after a further change of coordinates, we may assume that the tangent plane of $X$ at $P_1$ is $V(x_1)$; i.e. $d_{22}=0$ and $e_{22}=1$.

Since $P_1$ is an Eckardt point, three lines contained in $V(x_1)$ (the tangent plane of $X$ at $P_1$) need to pass through it. The first line is $\ell$, while the two other lines need to be of the form $\ell_\alpha:= P_1 \cup \{[1:0:t:\alpha] \mid t \in \mathbb{F}\}.$ It is not hard to see that $\ell_{\alpha} \subset X$ if and only if $a_2+d_{23}\alpha=0$ and $ a_3\alpha+d_{33}\alpha^2=0$. Since we need two possibilities for $\alpha$, we see that necessarily $a_2=d_{23}=0$ and $d_{33} \neq 0.$ Further replacing the coordinate $x_0$ by a suitable scalar multiple, we may assume that $d_{33}=1$. Replacing $x_0$ by $x_0+e_{33}x_1$, we may assume that $e_{33}=0$.
Hence we may assume that $F$ is of the form
\begin{equation}\label{eq:Eck1}
F
=
a_3x_0^2x_3+b_2x_0x_1x_2+b_3x_0x_1x_3+c_2x_1^2x_2+c_3x_1^2x_3
+x_0x_3^2
+x_1x_2^2+e_{23}x_1x_2x_3.
\end{equation}
But this implies that the point $P_2=[0:0:0:1]$ is a smooth point of $X$, since the partial derivative of $F$  w.r.t. $x_0$ does not vanish in $P_2$. Hence if $P_2$ was a weak Eckardt point, we conclude that $P_2$ actually is an Eckardt point of $X$.
\end{proof}

It is clear that if $X$ is a smooth cubic surface there is no distinction between weak Eckardt points and Eckardt points. The following lemma states how many weak Eckardt points a line on a cubic surface can contain. If $X$ is a smooth cubic surface, this result is already stated in Section 1.2 of \cite{DD}.

\begin{lemma}\label{Eckardt}
Let $X \subset \mathbb{P}^3$ be an irreducible cubic surface over a field $\mathbb{F}$ and $\ell$, $m$ be two skew lines on $X$. Assume that $\ell$ is not a double line. Then the number of weak Eckardt points on $\ell$ is at most $5$ if $\mathrm{char}(\mathbb{F})=2$ and at most $2$ otherwise.
\end{lemma}
\begin{proof}
Let $X=V(F)$ for a suitably chosen cubic polynomial $F$. Since $\ell$ is not a double line, at least one of the partial derivatives of $F$ remains nonzero when restricted to $\ell$. This implies that $\ell$ can contain at most two singular points, since the partial derivatives of $F$ are homogeneous polynomials of degree two. In view of Lemma \ref{weakEckardt}, we conclude that either $\ell$ contains up to two weak Eckardt points, or that any weak Eckardt point on $\ell$ actually is an Eckardt point. In the first case we are done. Therefore, from now on we assume that $\ell$ contains two Eckardt points $P_1$ and $P_2$ and that any additional weak Eckardt point on $\ell$ actually is an Eckardt point.

Reasoning as in the proof of Lemma \ref{weakEckardt}, we may assume without loss of generality that
$\ell=V(x_0,x_1)$, $m=V(x_2,x_3)$,  $P_1=[0:0:1:0],$ $P_2=[0:0:0:1]$ and $X=V(F)$, with $F$ of the form as in equation \eqref{eq:Eck1}.
However, since we now assume that $P_2$ is an Eckardt point as well, one can reason as for the point $P_1$ in the proof of Lemma \ref{weakEckardt} and obtain that $c_3=e_{23}=0.$ Hence $F$ is of the form
$$
F
=
a_3x_0^2x_3+b_2x_0x_1x_2+b_3x_0x_1x_3+c_2x_1^2x_2
+x_0x_3^2
+x_1x_2^2.
$$
If $\ell$ contains a third Eckardt point $P=[0:0:1:k]$ for some $k \in \mathbb{F}\backslash \{0\}$, then $\ell$ is contained in the tangent plane $\Pi$ of $X$ at $P$, which implies that $\Pi=V(x_1-\lambda x_0)$ for some $\lambda \in \mathbb{F}\backslash \{0\}.$ 
Restricting $F$ to $\Pi$ by eliminating $x_1$ we obtain that $F|_{\Pi}=x_0 Q(x_0,x_2,x_3),$ with $$Q(x_0,x_2,x_3):=(a_3+b_3 \lambda)x_0x_3+(b_2\lambda+c_2 \lambda^2 )x_0x_2+x_3^2+\lambda x_2^2.$$
Since $P$ is an Eckardt point, $Q(x_0,x_2,x_3)$ is the product of two degree one polynomials, each having $P$ as zero. Hence $P$ is a singular point of $V(Q)$, implying that $k^2+\lambda=0$ ($P$ a zero of $Q$), $(a_3+b_3\lambda)k+(b_2\lambda+c_2\lambda^2)=0$ ($P$ a zero of $\frac{\partial Q}{\partial x_0}$) and $2\lambda=0$ ($P$ a zero of $\frac{\partial Q}{\partial x_2}$). If $\mathrm{char}(\mathbb{F}) \neq 2$, then such a $P$ does not exist, since $\lambda \neq 0.$ Otherwise, we have $(a_3+b_3k^2)k+(b_2k^2+c_2k^4)=0,$ which can have at most three non-zero solutions for $k$, implying that $\ell$ contains at most five Eckardt points. This completes the proof.
\end{proof}

We now state the main result of this section.
\begin{theorem}\label{cubez}
Let $X \subset \mathbb{P}^3$ be a cubic surface defined over a field $\mathbb{F}$ containing two skew lines $\ell_1$ and $\ell_2$. Then one of the following holds:
\begin{enumerate}
\item[(a)] $X$ is reducible, or
\item[(b)] $X$ contains a double line, or
\item[(c)] $|\mathcal{T}_{\ell_1}| \le 5$ or $|\mathcal{T}_{\ell_2}| \le 5.$
\end{enumerate}
\end{theorem}
\begin{proof}
Let $X$ be an irreducible cubic surface not containing a double line. We need to show that $|\mathcal{T}_{\ell_1}| \le 5$ or $|\mathcal{T}_{\ell_2}| \le 5.$ If $|\mathbb{F}|\le 4$, this is trivial. Therefore we will assume from now on that $|\mathbb{F}|\ge 5.$ In view of Proposition \ref{cubenz} it is enough to show that $\mathcal{T}_{\ell_1} \neq \mathcal{B}(\ell_1)$ or $\mathcal{T}_{\ell_2} \neq \mathcal{B}(\ell_2)$. Suppose on the contrary that $\mathcal{T}_{\ell_1} = \mathcal{B}(\ell_1)$ and $\mathcal{T}_{\ell_2} =\mathcal{B}(\ell_2)$.

Any plane $\Pi \in \mathcal{B}(\ell_1)$ intersects the line $\ell_2$ in exactly one point, which we will denote by $P_\Pi$. Since $\mathcal{T}_{\ell_1} =\mathcal{B}(\ell_1)$, for each $\Pi \in \mathcal{B}(\ell_1)$, there exists a line $m_\Pi \subset \Pi$ such that $P_\Pi \in m_\Pi \subset X.$ Since $\ell_1$ and $\ell_2$ are skew, the lines $m_\Pi$ intersect the line $\ell_1$ in exactly one point. For a given point $P$ on $\ell_1$, at most two of the lines $m_\Pi$ may contain $P$. Indeed, such lines are contained in the intersection of $X$ and the plane containing $P$ and $\ell_2$, which contains two lines apart from $\ell_2$. Since by assumption $|\mathcal{B}(\ell_1)| = |\mathbb{F}|+1 \ge 6,$ this implies that we can find three mutually skew lines $m_1,m_2,$ and $m_3$ contained in $X$ meeting each of the lines $\ell_1$ and $\ell_2$ in exactly one point.

It is well known that there exists a unique, smooth quadric $Q$ containing the lines $m_1, m_2,$ and $m_3.$ Moreover, by construction of the lines $m_1,m_2,m_3$ we have $|Q \cap \ell_i| \ge 3$ for $i=1,2$, implying that both $\ell_1$ and $\ell_2$ are contained in $Q$. Since $X$ does not contain a double line, Proposition \ref{deg} (i.e., \cite[Cor.18.5 and 18.6]{H} ) implies that $Q \cap X$ contains a sixth line $\ell_3$. We claim that the lines $\ell_1, \ell_2,$ and $\ell_3$ are mutually skew. Indeed, if $\ell_1$ and $\ell_3$ would intersect, there would exist a plane $\Pi \in \mathcal{B}(\ell_2)$ containing $\ell_1 \cap \ell_3$. Then $\Pi \cap Q$ consists of two lines: $\ell_2$ and a line $n$ containing the point $\ell_1 \cap \ell_3.$ Then there would exist three distinct lines contained in $Q$, each passing through the point $\ell_1 \cap \ell_3.$ However, since $Q$ is smooth, this implies that the tangent plane of $Q$ at $\ell_1 \cap \ell_3$ intersects $Q$ in at least the lines $\ell_1, \ell_3$, and $n$. This in turn shows that $Q$ contains the tangent plane entirely, which gives a contradiction. Hence $\ell_1$ and $\ell_3$ are skew lines. A similar argument shows that $\ell_2$ and $\ell_3$ are skew.

Now consider a point $P$ on $\ell_3$. The plane passing through $P$ and $\ell_i$ ($i=1,2$), contains a line $n_{i,P} \subset X$ passing through $P$. Clearly $n_{i,P}$ is not one of the lines $\ell_1,\ell_2,$ or $\ell_3$, since these three lines are mutually skew, while $n_{i,P}$ intersects $\ell_3$ at $P$. If $n_{1,P}=n_{2,P},$ then this line intersects all of the lines $\ell_1,\ell_2,$ and $\ell_3.$ But then $n_{1,P}$ intersects $Q$ in at least three points and therefore is contained in $Q$. Since $n_{1,P}$ is contained in $X$ as well and $X \cap Q=\ell_1 \cup \ell_2 \cup \ell_3 \cup m_1 \cup m_2 \cup m_3,$ we conclude that if $n_{1,P}=n_{2,P},$ then $n_{1,P} \in \{m_1,m_2,m_3\}.$ In particular $n_{1,P}=n_{2,P}$ can occur for at most three points $P$ on $\ell_3.$ For the remaining points $P$ on $\ell_3$, we find at least three lines contained in $X$ passing through $P$, namely $n_{1,P}, n_{2,P}$, and $\ell_3$. Hence $\ell_3$ contains at least $|\ell_3|-3=|\mathbb{F}|-2$ weak Eckardt points. Since $|\mathbb{F}| \ge 5$, we directly find a contradiction in odd characteristic using Lemma \ref{Eckardt}. If $\mathrm{char}(\mathbb{F})=2$, the assumption that $|\mathbb{F}| \ge 5$, implies that $|\mathbb{F}| \ge 8$. Hence $\ell_3$ contains at least $6$ Eckardt points in this case. Again Lemma \ref{Eckardt} gives a contradiction.
\end{proof}
One can in fact show that in case (c) $|\mathcal{T}_{\ell_1}| \le 5$ and $|\mathcal{T}_{\ell_2}| \le 5,$ but  the current formulation is strong enough for our purposes.

\section{General results towards a proof of S{\o}rensen's conjecture }\label{sec:general}
The results in this section are oriented towards an attempt to resolve S{\o}rensen's conjecture in the general case. We show that by studying the incidence structures of lines and planes in $\PP^3 (\Fqt)$ with respect to the surface of degree $d$ in question, can lead to significant progress towards proving Conjecture \ref{So}. First, we mention that throughout this section $F$ will denote a nonzero homogeneous polynomial of degree $d$, where $2 \le d \le q$ in $\Fqt[x_0, x_1, x_2, x_3]$.
For a plane $\Pi$, which we always assume to be defined over $\Fqt$, we derive various upper bounds on $|\V(F) \cap V_2 \cap \Pi|$ depending on the line arrangements on $V(F)$ in $\Pi$. Also when considering a line contained in $\Pi$, we will always mean a line defined over $\Fqt.$

\begin{lemma}\label{plan1}
Let $\Pi$ (resp. $\ell \subset \Pi$) be  tangent to (resp. a generator of)  $V_2$ and suppose that $\Pi$ is not contained in $V(F).$ Then we have the following.
\begin{enumerate}
\item[(a)]  If $\ell \subset V(F)$, then
$$|\V(F) \cap V_2 \cap \Pi| \le dq^2 +1 \ \ \mathrm{and} \ \ |\V(F) \cap V_2 \cap (\Pi \setminus \ell)| \le (d-1)q^2.$$
\item[(b)] If $\ell \subset V(F)$ and $\ell$ is the only generator of $V_2$ contained in $V(F) \cap \Pi,$ then
$$|\V(F) \cap V_2 \cap \Pi| \le q^2 +(d-1)q+1 \ \ \mathrm{and} \ \ |\V(F) \cap V_2 \cap (\Pi \setminus \ell)| \le (d-1)q.$$
\item[(c)] If  $V(F) \cap \Pi$ does not contain any generators of $V_2,$ then
$$|\V(F) \cap V_2 \cap \Pi| \le d(q+1) \ \ \mathrm{and} \ \ |\V(F) \cap V_2 \cap (\Pi \setminus \ell)| \le dq.$$
\end{enumerate}
\end{lemma}

\begin{proof}
We begin by noting that $F|_{\Pi} \neq 0$.
\begin{enumerate}
\item[(a)] Suppose that $\Pi$ is a tangent to $V_2$ at a point $P$.
By Theorem \ref{linear}, we know that $\Pi$ contains exactly $q+1$ generators each passing through $P$. Since $\ell$ is one of the $q+1$ generators mentioned, clearly $P \in \ell$.
Theorem \ref{serre} implies that $|\V(F) \cap V_2 \cap \Pi| \le dq^2 +1$ and moreover this upper bound is attained if and only if $P \in \V(F)$ and $V(F)$ contains $d$ of the $q+1$ generators passing through $P$.
Further
$|\V(F) \cap V_2 \cap (\Pi \setminus \ell)| \le (d-1)q^2.$
\item[(b)]
If moreover,  $\ell$ is the only generator in the plane $\Pi$ that is contained in $V(F)$, then for any generator $\ell' \subset \Pi$ with $\ell \neq \ell'$, we have  $|\V(F) \cap \ell'| \le d$. This implies that $|\V(F) \cap V_2 \cap (\ell' \setminus \{P\})| \le d-1$. We have thus proved that $|\V(F) \cap V_2 \cap \Pi| \le q^2 + (d-1)q + 1$. We also deduce that, in the case when $V(F)$ contains only one generator $\ell$ in $\Pi$, then
$|\V(F) \cap V_2 \cap (\Pi \setminus \ell)| \le (d-1)q.$
\item[(c)] Since, for any generator $\ell \subset \Pi$, we have $|\V(F) \cap \ell| \le d$, we obtain  $|\V(F) \cap V_2 \cap \Pi| \le d(q+1)$. We further note that, if $\ell$ is a generator of $V_2$ contained in $\Pi$, then $V_2 \cap (\Pi \setminus \ell)$ is an affine curve of degree $q$, whereas $V(F) \cap (\Pi \setminus \ell)$ is an affine curve of degree $d$. Since they have no common components, we deduce from B\'ezout's theorem (or Proposition \ref{lac}) that $|\V(F) \cap V_2 \cap (\Pi \setminus \ell)| \le dq$.
\end{enumerate}
\end{proof}

\begin{lemma}\label{plan2}
Let $\Pi$ be a plane that is not a tangent to $V_2$ and $\ell$ be any line contained in $\Pi$. If $\Pi \not\subset V(F)$ then
$|\V(F) \cap V_2 \cap \Pi| \le d (q+1)$. Further, if $\ell \subseteq V(F)$ then $|\V(F) \cap V_2 \cap (\Pi \setminus \ell)| \le (d-1)(q+1).$
\end{lemma}
\begin{proof}
First, we note that $F|_{\Pi} \neq 0$.  Theorem \ref{linear} implies that  $V_2 \cap \Pi$ is a non-degenerate Hermitian curve and hence is irreducible, as noted in Section \ref{sec:prel}. Again, by applying Proposition \ref{lac} we see that, $|\V(F) \cap V_2 \cap \Pi| \le d (q+1).$ Now suppose that $V(F)$ contains a line $\ell \subseteq \Pi$. In this case $V(F) \cap (\Pi \setminus \ell)$ is an affine curve of degree $d-1$ and from Proposition \ref{lac} we deduce that $|\V(F) \cap V_2 \cap (\Pi \setminus \ell)| \le (d-1)(q+1).$
\end{proof}
\noindent We now make use of Lemma \ref{plan1} and \ref{plan2} to derive various upper bounds for $|\V(F) \cap V_2|$.
\begin{lemma}\label{nogen}
Suppose that $q>2$ and that $V(F)$ contains no generators of $V_2$. Then
$$|\V(F) \cap V_2| \le d(q^3 + q + 1) < dq^3 + (d-1)q^2 + 1.$$
\end{lemma}

\begin{proof}
Evidently, if $\ell$ is a generator of $V_2$, then $|\V(F) \cap \ell| \le d$.   For any $\Pi \in \mathcal{B}(\ell)$, as noted in Lemma \ref{plan1}(c), we have $|\V(F) \cap V_2 \cap (\Pi \setminus \ell)| \le dq$. Hence,
\begin{align*}
|\V(F) \cap V_2|&  \le  \sum_{\Pi \in \mathcal{B}(\ell)} |\V(F) \cap V_2 \cap (\Pi \setminus \ell)| + |\V(F) \cap V_2 \cap \ell| \\
& \le dq(q^2 + 1) + d \\
&= dq^3 + dq + d < dq^3 + (d-1)q^2 + 1.
\end{align*}
This completes the proof.
\end{proof}

\begin{remark} \normalfont
Note that in the situation of Lemma \ref{nogen}, the set $V(F) \cap V_2$ is an algebraic curve of degree at most $d(q+1)$ not containing any lines. For any $q$ (including $q=2$), the bound $|\V(F) \cap V_2| \le dq^3 + (d-1)q^2 + 1$ then follows directly by applying a bound on the number of rational points on a  curve containing no lines due to Homma \cite[Theorem 1.1]{Homma}. Lemma \ref{nogen} shows that, in this case, apparently a better bound on $|\V(F) \cap V_2|$ is possible for $q>2$.
\end{remark}

\begin{lemma}\label{noskew}
Suppose that $V(F)$ contains a generator $\ell$ of $V_2$, but contains no two skew generators. Further assume that $V(F)$ does not contain any plane. Then
$$|\V(F) \cap V_2| \le(d-1)q^3 + dq^2 + 1 < dq^3 + (d-1)q^2 + 1.$$
\end{lemma}

\begin{proof}
Since $V(F)$ does not contain any plane, it is evident that $F|_{\Pi}$ is a nonzero polynomial for any plane in $\PP^3(\Fqt)$. The proof is divided into two cases.

\textbf{Case 1:} Suppose that $V(F)$ contains another generator $\ell'$. By hypothesis, the generators $\ell$ and $\ell'$ are contained in a plane $\Pi$. Let $\Pi' \in \mathcal{B}(\ell)$ with $\Pi' \neq \Pi$ and suppose that $\ell_1$ is a generator of $V_2$ contained in $\Pi'.$ If $\ell_1 \neq \ell$, then it follows from Theorem \ref{linear} and Proposition \ref{descbook} (a) that $\ell'$ and $\ell_1$ are skew lines.  By the hypothesis on $V(F)$, we have $\ell_1 \not\subset V(F)$.
Thus, using Lemma \ref{plan1} (a) and (b), we have $|\V(F) \cap V_2 \cap \Pi| \le dq^2 + 1$ and  $|\V(F) \cap V_2 \cap (\Pi' \setminus \ell)| \le (d-1)q$. This proves that,
\begin{align*}
|\V(F) \cap V_2| &= |\V(F) \cap V_2 \cap \Pi| + \sum_{\Pi' \neq \Pi} |\V(F) \cap V_2 \cap (\Pi' \setminus \ell)| \\
&\le dq^2 + 1 + q^2 (d-1) q \\
&=(d-1)q^3 + dq^2 + 1 < dq^3 + (d-1)q^2 + 1.
\end{align*}

\textbf{Case 2:} Suppose that $V(F)$ contains no other generators. Using Lemma \ref{plan1} (b), we deduce that $|\V(F) \cap V_2 \cap (\Pi \setminus \ell)| \le (d-1)q$ for any $\Pi \in \mathcal{B}(\ell)$. Hence,
\begin{align*}
|\V(F) \cap V_2| &\le \sum_{\Pi \in \mathcal{B}(\ell)} |\V(F) \cap V_2 \cap (\Pi \setminus \ell)| + |\ell| \\
& \le (q^2 + 1) (d-1)q + q^2 + 1 \\
& = (d-1)q^3 + q^2 + (d-1)q + 1 \\
& = (dq^3 + (d-1)q^2 + 1) - (q^3 + (d-2) q^2 - (d-1)q) < dq^3 + (d-1)q^2 + 1.
\end{align*}
This completes the proof.
\end{proof}

\noindent Having investigated cases where $V(F)$ does not contain a plane, in the remainder of this section we move our attention to cases where $V(F)$ contains at least one plane in $\PP^3 (\Fqt)$.

\begin{lemma}\label{nontan}
Suppose that $V(F)$ contains a plane which is not  tangent to $V_2$. If  Conjecture \ref{So} is true for polynomials of degree at most $d - 1$, then
$|\V(F) \cap V_2| \le dq^3 + (d-1)q^2 + 1.$ Further, for $d > 2$ we have
$|\V(F) \cap V_2| \le dq^3 + (d-1)q^2 - (d-2)q + 2 <dq^3 + (d-1)q^2 + 1.$
\end{lemma}

\begin{proof}
\textbf{Case 1:} Let $d = 2$. Then $V(F)$ is a union of two planes, say $\Pi_1$ and $\Pi_2$. We assume without loss of generality that, $\Pi_1$ is a plane not tangent to $V_2$.  From Theorem \ref{linear}, we see that $|V_2 \cap \Pi_1| = q^3 + 1$, while $|V_2 \cap \Pi_2| \le q^3 + q^2 + 1$. Further, from Lemma \ref{lines} we have, $|V_2 \cap \Pi_1 \cap \Pi_2| \ge 1$. This shows that, $|\V(F) \cap V_2 | \le 2q^3 + q^2 + 1$.

\textbf{Case 2:} Assume that $d \ge 3$. We may write $F = HG$, where $H, G \in \Fqt[x_0, x_1, x_2, x_3]$, with $\deg H = 1$ and $\deg G \ge 2$.
Further, $V(H)$ is a not a tangent to $V_2$. Theorem \ref{linear} shows that $|\V(H) \cap V_2| = q^3 + 1$. Further, since Conjecture \ref{So} is assumed to be true for polynomials of degree at most $d-1$, we see that $|\V(G) \cap V_2| \le (d-1)(q^3 + q^2 - q) + q + 1$. Consequently,
\begin{align*}
|\V(F) \cap V_2| & \le |\V(G) \cap V_2| + |\V(H) \cap V_2| \\
&\le (d-1)(q^3 + q^2 -q) + q + 1 + q^3 + 1 \\
&= dq^3 + (d-1)q^2 - (d-2)q + 2\\
&= dq^3 + (d-1)q^2 + 1 - ((d-2)q -1) < dq^3 + (d-1)q^2 + 1.
\end{align*}
This completes the proof.
\end{proof}

\begin{lemma}\label{gentang}
Suppose that, $V(F)$ contains a generator $\ell$ and at least two planes in $\mathcal{B}(\ell)$. Further assume that the conjecture is true for polynomials of degree at most $d-1$. Then
$$|\V(F) \cap V_2| \le dq^3 + (d-1)q^2 + 1.$$
\end{lemma}

\begin{proof}
Let $\Pi_1, \dots, \Pi_s \in \mathcal{B} (\ell)$ such that $\Pi_k \subseteq V(F)$ for all $k = 1, \dots, s$. Since, by assumption, $V(F)$ contains at least two planes from $\mathcal{B}(\ell)$, we have $2 \le s \le d$. First assume that $s = d$. Then $V(F) = \Pi_1 \cup \cdots \cup \Pi_d$ and consequently, $|\V(F) \cap V_2| = dq^3 + q^2 + 1$. So from now on we may assume that $s < d$.

\textbf{Case 1:} Let $(d, s) = (3, 2)$. We may write $V(F) = \Pi_1 \cup \Pi_2 \cup \Pi_3$, where $\Pi_1 \cap \Pi_2 = \ell$ and $\ell \not\subset \Pi_3$. We note that $|(\Pi_1 \cup \Pi_2) \cap V_2| = 2q^3 + q^2 + 1$, whereas from Theorem \ref{linear} we see that $|\Pi_3 \cap V_2| \le q^3 + q^2 + 1$. Since, any line intersects $V_2$ in at least one point (see Lemma \ref{line}), we see that $|(\Pi_1 \cup \Pi_2) \cap \Pi_3 \cap V_2| \ge 1.$ This implies that,
\begin{align*}
|\V(F) \cap V_2| &= |(\Pi_1 \cup \Pi_2) \cap V_2| + |\Pi_3 \cap V_2| - |(\Pi_1 \cup \Pi_2) \cap \Pi_3 \cap V_2| \\
&\le 2q^3 + q^2 + 1 + q^3 + q^2 + 1 - 1 \\
&= 3q^3 + 2q^2 + 1.
\end{align*}
\textbf{Case 2:} Let $(d, s) \neq (3, 2)$. We may write $F = H_1 \dots H_s G$, where $H_1, \dots, H_s, G \in \Fqt[x_0, x_1, x_2, x_3]$ with $\deg H_1 = \dots =\deg H_s = 1$ and $\deg G = d-s$. Further, we write that $\Pi_1 = V(H_1), \dots,  \Pi_s = V(H_s)$ and note that $\ell = \Pi_1 \cap \dots \cap \Pi_s$.
We have,
$$|(\Pi_1 \cup \dots \cup \Pi_s) \cap V_2| = sq^3 + q^2 + 1 \ \ \  \mathrm{and} \ \ \ |\V(G) \cap V_2| \le (d-s)(q^3 + q^2 - q) + q + 1.$$
While the first assertion above follows trivially, the second one is a direct consequence of the hypothesis that the conjecture is true for polynomials of degree at most $d-1$.
This implies that
\begin{align*}
|\V(F) \cap V_2| &\le |(\Pi_1 \cup \dots \cup \Pi_s) \cap V_2| + |\V(G) \cap V_2| \\
& \le sq^3 + q^2 + 1 + (d-s)(q^3 + q^2 - q) + q + 1 \\
& = dq^3 + (d-s + 1)q^2 - (d-s-1)q + 2 \\
& = dq^3 + (d-1)q^2 + 1 + \left( - (s-2) q^2 - (d-s - 1) q + 1\right) \le  dq^3 + (d-1)q^2 + 1.
\end{align*}
Note that for $2 \le s \le d-1$ the quantity $(s-2) q^2 + (d-s - 1) q$ is nonnegative and zero if and only if $d = 3$ and $s = 2$. The last inequality now follows since $(d, s) \neq (3, 2)$.
\end{proof}


\begin{remark} \normalfont
To prove S{\o}rensen's conjecture for $d=2$, Edoukou has made use of the classification of quadric surfaces and proved the conjecture for each class of quadrics. We remark that, if we use the results in this section then the classification of quadric surfaces will not be needed anymore. Indeed, if a quadric surface is reducible over $\Fqt$, then it is given by union of two planes. In the case when one of the planes is not tangent to $V_2$, we can apply Lemma \ref{nontan} with the fact that the conjecture is true for $d=1$, to get the desired inequality. The other case of a reducible quadric surface occurs if the surface is union of two tangent planes. As already noted in Corollary \ref{int}, the two tangent planes can intersect at a secant or at a generator. In the case when they intersect at a secant the upper bound in S{\o}rensen's conjecture is attained (see Proposition \ref{attained}) while in the latter case, the desired inequality is easily derived from Lemma \ref{gentang}. This leads us to the case when the quadric surface is irreducible over $\Fqt$. For irreducible quadrics containing no generators or containing a generator but no two skew generators, we could obtain the desired inequalities by using Lemma \ref{nogen} and \ref{noskew} respectively. Finally, after a linear change of variables, any quadric containing two skew lines is  given by an equation of the form $x_0 x_1 + x_2 x_3 = 0$, which is a hyperbolic quadric. This case was proved in \S 5.2.1 in \cite{E}. Although this proof does not necessarily shorten or essentially simplify the proof of Edoukou's theorem (since, after all, the most nontrivial part of his proof lies in \S 5.2.1 of \cite{E}), it may give an idea of generalizing the proof for larger values of $d$.
\end{remark}

\section{Proof of S{\o}rensen's conjecture for cubic surfaces}\label{sec:mt}
In this section, we will make use of the results that we have derived in Sections \ref{sec:cubic} and \ref{sec:general} to show that Conjecture \ref{So} is true for $d=3$ whenever $q \ge 3$. We begin with the following:
\begin{proposition}\label{reducible}
Let $F \in \Fqt[x_0, x_1, x_2, x_3]$ be a reducible cubic. Further assume that any linear factor of $F$ corresponds to a  plane tangent to $V_2$. Then either
$$|\V(F) \cap V_2| = 3(q^3 + q^2 - q) + q + 1$$
or
$$|\V(F) \cap V_2| \le \max\{3q^3 + 2q^2 + 2, 3(q^3 + q^2 - q) + 1\} < 3(q^3 + q^2 - q) + q + 1.$$
In particular, if $q \ge 4$  and  $|\V(F) \cap V_2| < 3(q^3 + q^2 - q) + q + 1$, then $|\V(F) \cap V_2| \le 3(q^3 + q^2- q)  + 1$.
\end{proposition}
\begin{proof}
Since $F$ is reducible, we may write $F = H Q$, where $H, Q \in \Fqt[x_0, x_1, x_2, x_3]$ with $\deg H = 1$ and $\deg Q =2$. Further, by assumption, $V(H)$ is a plane that is tangent to $V_2$. This immediately proves that $|\V(H) \cap V_2| = q^3 + q^2 + 1$.

\textbf{Case 1:} Let $Q$ be a homogeneous quadratic polynomial such that $|\V(Q) \cap V_2| < 2(q^3 + q^2 - q) + q + 1$. In this case, Theorem \ref{Ed} implies that $|\V(Q) \cap V_2| \le 2q^3 + q^2 + 1$. This shows that,
$$|\V(F) \cap V_2| \le |\V(H) \cap V_2| + |\V(Q) \cap V_2| \le 3q^3 + 2q^2 + 2 < 3(q^3 + q^2 -q) + q + 1,$$
the last inequality follows since $q \ge 3$. We remark that, if $q > 3$, then $|\V(F) \cap V_2| \le 3(q^3 + q^2 - q)$.

\textbf{Case 2:} Let $Q$ be a homogeneous quadratic polynomial such that $|\V(Q) \cap V_2| = 2(q^3 + q^2 - q) + q + 1$. Then $V(Q)$ is a union of two planes $\Pi_1, \Pi_2$, both tangent to $V_2$,  such that the line $\ell = \Pi_1 \cap \Pi_2$ intersects $V_2$ at $q+1$ points. We write $\Pi_0 = V(H)$. Note that, if $\ell \subseteq \Pi_0$ then  Proposition \ref{attained} implies that $|\V(F) \cap V_2| = 3(q^3 + q^2 - q) + q + 1$. We may thus assume that $\ell \not\subseteq \Pi_0$. If $\Pi_0$ intersects $\Pi_1$ or $\Pi_2$ at a generator, then Lemma \ref{gentang} applies (since Conjecture \ref{So} is true for $d=2$) and the proposition follows. Thus, in view of  Corollary \ref{int}, it is enough to prove the proposition in the case when
\begin{equation}\label{next}
|\Pi_0 \cap \Pi_1 \cap V_2| = q+1 \ \ \ \ \mathrm{and} \ \ \ \ |\Pi_0 \cap \Pi_2 \cap V_2| = q+1.
\end{equation}
Applying the above conditions and using inclusion-exclusion principle, we have
\begin{align*}
|\V(F) \cap V_2| &= \sum_{i=0}^2 |\Pi_i \cap V_2| - \sum_{i<j}|\Pi_i \cap \Pi_j \cap V_2| + |\Pi_0 \cap \Pi_1 \cap \Pi_2 \cap V_2| \\
&\le 3(q^3 + q^2 + 1) - 3(q+1) + 1 \\
&= 3(q^3 + q^2 - q) + 1.
\end{align*}
This completes the proof.
\end{proof}

In Section \ref{sec:cubic} we came across a case where the cubic surface in question may contain a double line. The following Lemma shows that the upper bound of Conjecture \ref{So} holds in this case.

\begin{lemma}\label{doubleline}
Let $F \in \Fqt[x_0, x_1, x_2, x_3]$ be an irreducible cubic. Suppose that $V(F)$ contains a double line $\ell$. Then $|\V(F) \cap V_2| \le 3q^3 + 2q^2 + 1.$
\end{lemma}

\begin{proof}
Since $F$ is irreducible in $\Fqt[x_0, x_1, x_2, x_3]$, for every plane $\Pi \in \PP^3 (\Fqt)$, the polynomial $F|_{\Pi}$ is a nonzero cubic.  Since $\ell$ is a double line, for any $\Pi \in \mathcal{B}(\ell)$, we have  $V(F) \cap \Pi = \ell \cup \ell_{\Pi}$ for some line $\ell_{\Pi}$ contained in $\Pi$.

\textbf{Case 1:} The line $\ell$ is a tangent line.
Let $\Pi$ denote the unique plane containing $\ell$ which is tangent to $V_2$. Then $|\V(F) \cap V_2 \cap \Pi| \le q^2 + 1$ and equality holds if and only if $\ell_{\Pi}$ is a generator. Since $|\V(F) \cap V_2 \cap \ell| = 1$, we conclude that  $|\V(F) \cap V_2 \cap (\Pi \setminus \ell)| \le q^2$.
Let $\Pi_1, \dots, \Pi_{q^2}$ denote the remaining planes containing $\ell$. It follows from Proposition \ref{descbook} (b) that the $\Pi_j$-s are not tangent to $V_2$.  Consequently, for each $j =1. \dots, q^2$, we see that $V_2 \cap \Pi_j$ is a non-degenerate Hermitian curve which intersects the line $\ell_{\Pi_j}$ at most $q+1$ points outside $\ell$. In other words, $|\V(F) \cap V_2 \cap (\Pi_j \setminus \ell)| \le q + 1$ for each $j =1, \dots, q^2$. Hence,
\begin{align*}
|\V(F) \cap V_2| &= |\V(F) \cap V_2 \cap (\Pi \setminus \ell)| + \sum_{j=1}^{q^2} |\V(F) \cap V_2 \cap (\Pi_j \setminus \ell)| + |\V(F) \cap V_2 \cap \ell|\\
&\le q^2 + q^2 (q+1) + 1 \\
&= q^3 + 2q^2 + 1 < 3q^3 + 2q^2 + 1.
\end{align*}

\textbf{Case 2:} The line $\ell$ is a secant line.
Let $\Pi_1, \dots, \Pi_{q+1}$ denote the $q+1$ planes containing $\ell$ that are tangent to $V_2$ and $\Pi_{q+2}, \dots, \Pi_{q^2 + 1}$ denote the remaining planes containing $\ell$ that are not tangent to $V_2$.
For $j=1, \dots, q+1$, we claim that $|\V(F) \cap V_2 \cap (\Pi_j \setminus \ell)| \le q^2$. The assertion follows trivially if $\ell_{\Pi_j}$ is not a generator, since $\ell_{\Pi_j}$ intersects each of the $q+1$ generators in $\Pi_j$ at a single point. In the case when $\ell_{\Pi_j}$ is a generator, we have $\V(F) \cap V_2 \cap \ell_{\Pi_j} = \ell_{\Pi_j}$ and the point of intersection of $\ell$ and $\ell_{\Pi_j}$ belongs to $V_2$. This proves the claim. Moreover, using a similar argument as in Case 1, we see that  $|\V(F) \cap V_2 \cap (\Pi_j \setminus \ell)| \le q+1$ for $j = q+2, \dots, q^2 + 1$. Thus,
\begin{align*}
|\V(F) \cap V_2| &=\sum_{j=1}^{q+1} |\V(F) \cap V_2 \cap (\Pi_j \setminus \ell)| + \sum_{j=q+2}^{q^2 +1} |\V(F) \cap V_2 \cap (\Pi_j \setminus \ell)| + |\V(F) \cap V_2 \cap \ell|\\
&\le (q+1)q^2 + (q^2 -q)(q+1) + q+1 \\
&= (q+1)(2q^2 -q + 1)= 2q^3 + q^2 + 1 < 3q^3 + 2q^2 + 1.
\end{align*}

\textbf{Case 3:} The line $\ell$ is a generator. First assume that there exists $\Pi \in \mathcal{B}(\ell)$ such that $\ell_{\Pi}$ is also a generator. Then
\begin{equation*}
|\V(F) \cap V_2 \cap \Pi| =
\begin{cases}
2q^2 + 1,  &\mathrm{if} \ \ \ell \neq \ell_{\Pi} \\
q^2 + 1, &\mathrm{if} \ \ \ell = \ell_{\Pi}
\end{cases}
\end{equation*}
 Since $\ell$ is an irreducible component of $V(F) \cap V_2$ and $V(F)$ is singular at every point on $\ell$, Proposition \ref{coprime} implies that $V(F) \cap V_2$ is an equidimensional variety of dimension $1$ and degree at most $3(q+1) -1$. Since $\Pi$ contains two components, namely $\ell$ and $\ell_{\Pi}$, of $V(F) \cap V_2$ of degree $1$ each, we see that $V(F) \cap V_2 \cap \Pi^C$ is an affine variety of degree at most $3q$ if $\ell \neq \ell_{\Pi}$ and at most $3q+1$ if $\ell = \ell_{\Pi}$. Further, the equidimensionality of $V(F) \cap V_2$ readily implies the equidimensionality of $V(F) \cap V_2 \cap \Pi^C$ as an affine variety. Applying Proposition \ref{lac}, we obtain
\begin{equation*}
|\V(F) \cap V_2 \cap \Pi^C| \le
\begin{cases}
3q^3, &\mathrm{if} \ \ \ell \neq \ell_{\Pi} \\
3q^3 + q^2 &\mathrm{if} \ \ \ell = \ell_{\Pi}
\end{cases}
\end{equation*}
Thus,
$$|\V(F) \cap V_2| = |\V(F) \cap V_2 \cap \Pi| + |\V(F) \cap V_2 \cap \Pi^C| \le 3q^3 + 2q^2 + 1.$$
Now suppose that for each $\Pi \in \mathcal{B}(\ell)$, the line $\ell_{\Pi}$ is either a tangent or a secant to $V_2$. Then
$$|\V(F) \cap V_2| = |\ell| + \sum_{\Pi \in \mathcal{B}(\ell)} |V_2 \cap \ell_{\Pi}|
\le q^2 + 1 + (q^2 + 1) (q+1)
= q^3 + 2q^2 + q + 2 < 3q^3 + 2q^2 + 1. $$
This completes the proof.
\end{proof}
\noindent We are now ready to state and prove the main theorem of this paper.
\begin{theorem}\label{mt}
Let $q \ge 3$. For any homogeneous cubic $F \in \Fqt[x_0, x_1, x_2, x_3]$ we have $|\V(F) \cap V_2| \le 3(q^3 + q^2 - q) + q + 1.$ Moreover, if $q \ge 4$, then 
$$|\V(F) \cap V_2| = 3(q^3 + q^2 - q) + q + 1 \ \ \ \ \mathrm{or} \ \ \ \ |\V(F) \cap V_2| \le 3(q^3 + q^2 - q) + 1.$$
\end{theorem}

\begin{proof}
\textbf{Case 1:} $F$ is reducible. If $V(F)$ contains a  plane not tangent to $V_2$, Lemma \ref{nontan} applies and shows that $|\V(F) \cap V_2| < 3q^3 + 2q^2 + 1$.
Otherwise, Proposition \ref{reducible} applies. 

\textbf{Case 2:} $F$ is irreducible. If $V(F)$ contains no generator, then Lemma \ref{nogen} shows that $|\V(F) \cap V_2| < 3q^3 + 2q^2 + 1$. So we may assume that $V(F)$ contains a generator. If $V(F)$ contains no two skew generators then from Lemma \ref{noskew} we see that $|\V(F) \cap V_2| < 3q^3 + 2q^2 + 1$. This leads us to investigate the case when $V(F)$ contains two skew generators. From the results of Section \ref{sec:cubic}, we see that such cubics satisfy one of the following two properties:
\begin{enumerate}
\item[(a)] $V(F)$ contains a double line.
\item[(b)] There exists a generator $\ell$ contained in $V(F)$, such that $|\mathcal{T}_\ell| \le 5$.
\end{enumerate}
For (a), Lemma \ref{doubleline} applies and shows that $|\V(F) \cap V_2| \le 3q^3 + 2q^2 + 1$. Now suppose that (b) holds. 
Let $\Pi_1, \dots, \Pi_s$ be $s$ planes containing $\ell$ such that $V(F) \cap \Pi_j$ is a union of lines. Further let $0\le t \le s$ be such that $V(F) \cap \Pi_j$ is a union of lines forming a triangle for $1 \le j \le t$, that is to say that the three lines do not share a common point. Then for $t+1 \le j \le s$, the intersection $V(F) \cap \Pi_j$ is a union of three lines passing through a common point $P_j$ on $\ell$. This implies that $P_j$ is a weak Eckardt point for $t+1 \le j \le s.$ We have
\begin{equation*}
|\V(F) \cap V_2 \cap (\Pi_j \setminus \ell)| \le q^2+q-1 \ \makebox{for $1 \le j \le t$}
\end{equation*}
and
\begin{equation}\label{eq:Eckardtpoints}
|\V(F) \cap V_2 \cap (\Pi_j \setminus \ell)| \le 2q^2 \ \makebox{for $t+1 \le j \le s$}.
\end{equation}
Finally using Lemma \ref{plan1}(b) we obtain, $|\V(F) \cap V_2 \cap (\Pi \setminus \ell)| \le 2q$ for the remaining planes $\Pi \in \mathcal{B}(\ell).$ In this case $V(F) \cap \Pi$ is a union of $\ell$ and an irreducible conic.

This shows that,
\begin{align}\label{less_crude_est}
|\V(F) \cap V_2|  & =|\ell| + \sum_{i=1}^s |\V(F) \cap V_2 \cap (\Pi_j \setminus \ell)| +  \sum_{\substack{\Pi \in \mathcal{B}(\ell) \\ \Pi \neq \Pi_j}}|\V(F) \cap V_2 \cap (\Pi \setminus \ell)| \notag\\
& \le q^2 + 1 + t(q^2+q-1)+2(s-t)q^2 + (q^2 + 1 - s) 2q\notag\\ 
&= 2q^3 + (2s-t + 1)q^2 - (2s-t-2)q + 1-t.
\end{align}
Since $0 \le t \le s \le 5$, a simple calculation shows that if $q \ge 8$, then $2q^3 + (2s-t + 1)q^2 - (2s-t-2)q + 1-t < 3(q^3 + q^2 - q)$.
Note that that $s-t \le 2$ in odd characteristic by Lemma \ref{Eckardt}. Therefore, we obtain that for $q \ge 5$ and odd: $2q^3 + (2s-t + 1)q^2 - (2s-t-2)q + 1-t  < 3(q^3 + q^2 - q)$.

Now we look at the remaining cases $q=3$ and $q=4$. Note that equality holds in equation \eqref{eq:Eckardtpoints} if and only if all lines in $V(F) \cap \Pi_j$ are generators of $V_2$. If this is not the case, then the estimate can be improved to $|\V(F) \cap V_2 \cap (\Pi_j \setminus \ell)| \le q^2+q.$ Hence, if none of the planes $\Pi_{t+1},\dots,\Pi_s$ intersect $X$ in three generators, then the estimate in \eqref{less_crude_est} can be improved to $$|\V(F) \cap V_2| \le 2q^3 + (s +1)q^2 - (s-2)q + 1-t.$$
Since $s \le 5$ and $s-t \le 2$ if $q$ is odd, we find for $q=3$ and $q=4$ that $|\V(F) \cap V_2| < 3(q^3+q^2-q)$. Therefore we may now assume that there exists a plane $\Pi \in \mathcal{B}(\ell)$ such that $V(F) \cap \Pi$ contains three generators.  Define $$g=|\{m \subset V(F) \ \makebox{a generator} \mid m \neq \ell \ \makebox{and} \ m \subset \Pi \ \makebox{for some $\Pi \in \mathcal{B}(\ell)$}  \}|.$$ 
Since the $g$ generators contained  in $\Pi \in \mathcal{B}(\ell)$ all intersect $\ell$, the total number of points on either $\ell$ or one of these $g$ generators equals $q^2+1+gq^2$. For any $\Pi \in \mathcal{B}(\ell)$, the remaining number of points of $V(F) \cap \Pi$ can be estimated by $2q$. However, if $\Pi$ contains one, resp. two of the $g$ generators, this estimate can be improved by $q$, resp. $2q$. This leads to the inequality 
$$|\V(F) \cap V_2| \le q^2+1+gq^2+(2(q^2 + 1)-g)q=(q^2+1)(2q+1)+g(q^2-q).$$ This implies that $|\V(F) \cap V_2| \le 3(q^3 + q^2 - q)+1$ if $g\le q+2.$ For $q=4$, the inequality is strict. Therefore, we may now assume that $V(F) \cap \Pi$ contains three generators $\ell_1:=\ell$, $\ell_2$ and $\ell_3$ such that for all $i$, there exist $q+1$ generators distinct from $\ell_1,\ell_2$ and $\ell_3$, that intersect $\ell_i$. These $3(q+1)$ generators are all mutually distinct, since otherwise there would exist a triple of generators forming a triangle. However, this is not possible, since co-planar generators intersect in a common point. Therefore, counting $\ell_1,\ell_2,\ell_3$ as well, we have found in total $3(q+2)$ distinct generators contained in $V(F) \cap V_2.$ However, this is not possible, since the degree of $V(F) \cap V_2$ is at most $3(q+1).$ This completes the proof.
%
\end{proof}

\begin{corollary}\label{class}
Let $q \ge 3$. The cubic surfaces defined over $\Fqt$ that attain the upper bound in Conjecture \ref{So} are given by a  union of three planes tangent to $V_2$ containing a common line which intersects $V_2$ at precisely $q+1$ points.
\end{corollary}

\begin{proof}
Follows immediately from  the proofs of Proposition \ref{reducible} and Theorem \ref{mt}.
\end{proof}

\noindent Note that, Theorem \ref{mt} and Corollary \ref{class} prove Conjecture \ref{So} for $d=3$. 
For $q \ge 4$ analyzing the proofs of Proposition \ref{reducible} and Theorem \ref{mt} actually yields a classification of all the homogeneous cubics $F \in \Fqt[x_0, x_1, x_2, x_3]$ such that $|\V(F) \cap V_2| = 3(q^3 + q^2 - q) +1$.

\begin{corollary}\label{second}
Let $q \ge 4$. Let $F \in \Fqt[x_0, x_1, x_2, x_3]$ be a homogeneous cubic such that $|\V(F) \cap V_2| = 3(q^3 + q^2 - q)  +1$. Then $V(F)$ is a union of three distinct  planes $\Pi_1, \Pi_2, \Pi_3$ tangent to $V_2$ satisfying
\begin{enumerate}
\item[(a)] For $i \neq j$, the line $\ell_{ij} = \Pi_i \cap \Pi_j$ is a secant.
\item[(b)] $|\Pi_1 \cap \Pi_2 \cap \Pi_3 \cap V_2| = 1$.
\end{enumerate}
\end{corollary}


\noindent  The following theorem guarantees the existence of a homogeneous cubic in $\Fqt[x_0, x_1, x_2, x_3]$ satisfying the assertions of Corollary \ref{second}.
\begin{theorem}\label{2nd}
There exist three tangent planes $\Pi_1, \Pi_2, \Pi_3$ to $V_2$ such that
\begin{enumerate}
\item[(a)] For $i \neq j$, the line $\ell_{ij} = \Pi_i \cap \Pi_j$ is a secant.
\item[(b)] $|\Pi_1 \cap \Pi_2 \cap \Pi_3 \cap V_2| = 1$.
\end{enumerate}
In particular, there exists a homogeneous cubic $F \in \Fqt[x_0, x_1, x_2, x_3]$ such that $|\V(F) \cap V_2| = 3(q^3 + q^2 - q) +1$. Consequently, the second highest number of points of intersection of cubic surfaces defined over $\Fqt$ and $V_2$ is given by $3(q^3 + q^2 - q) + 1$.
\end{theorem}

\begin{proof}
Let $P \in V_2$ and $\Pi$ be tangent to $V_2$ at $P$. As noted in Theorem \ref{linear}, the plane $\Pi$ intersects $V_2$ at $q+1$ generators passing through $P$. We fix three such generators $\ell_1, \ell_2, \ell_3$ and a line $\ell$ passing through $P$ that is not contained in $\Pi$. Remark \ref{linetp} shows that the line $\ell$ is a secant. Choose $\Pi_1$ to be the unique plane containing $\ell$ and $\ell_1$ and $\Pi_2$ the unique plane containing $\ell$ and $\ell_2$.  Clearly, $\Pi_1 \cap \Pi_2 = \ell$ and we define $\ell_{12} := \ell$. Since $\Pi_1, \Pi_2$ contain the generators $\ell_1, \ell_2$ respectively, they are tangent to $V_2$ at some point other than $P$. Fix a line $\ell_{13}$ contained in $\Pi_1$ passing through $P$ with the property that $\ell_{13} \neq \ell_1$. Since $\ell_{13}$ does not lie completely in $\Pi$, it is a secant line (see Remark \ref{linetp}). Define $\Pi_3$ to be the unique plane that contains $\ell_{13}$ and $\ell_3$. Clearly $\ell_{13} = \Pi_1 \cap \Pi_3$.  Since $\Pi_3$ contains a generator $\ell_3$ it is tangent to $V_2$ but since $\ell_{13} \not\subset \Pi$ we see that $\Pi_3 \neq \Pi$. Further, it is clear that $\Pi_1, \Pi_2, \Pi_3$ are distinct. We define $\ell_{23} = \Pi_2 \cap \Pi_3$. Note that $P \in \ell_{23}$ since $P \in \Pi_2$ and $P \in \Pi_3$. Since $\Pi_2, \Pi_3$ are tangent planes, it follows from Corollary \ref{int} (a)  that $\ell_{23}$ could either be a generator or a secant. If $\ell_{23}$ would be a generator, then $\Pi_3$ would contain two generators, namely $\ell_3$ and $\ell_{23}$, both passing through $P$, which would imply that $\Pi_3 = \Pi$, leading to a contradiction. Thus $\ell_{23}$ is a secant line. Hence $\Pi_1, \Pi_2, \Pi_3$ satisfy assertion (a) and since $\{P\} = \Pi_1 \cap \Pi_2 \cap \Pi_3 \cap V_2$ the second assertion follows as well.

Clearly, the planes $\Pi_1, \Pi_2, \Pi_3$ are defined over $\Fqt$ and therefore are given by the zero set of linear polynomials $H_1, H_2, H_3 \in \Fqt[x_0, x_1, x_2, x_3]$. Take $F = H_1 H_2 H_3$. Note that,
\begin{align*}
|\V(F) \cap V_2| &= \sum_{i=1}^3 |\Pi_i \cap V_2| - \sum_{i < j} |\Pi_i \cap \Pi_j \cap V_2| + |\Pi_1 \cap \Pi_2 \cap \Pi_3 \cap V_2| \\
&=3(q^3 + q^2 + 1) - 3 (q+1) + 1 = 3(q^3 + q^2 - q) + 1.
\end{align*}
This completes the proof.
\end{proof}

\noindent In \cite{ELX} several conjectures were made related to the intersection of a hypersurface of degree $d$ and a Hermitian variety. Specialized to the case of non-degenerate Hermitian surfaces,  Conjecture 2(i) in \cite{ELX} can be rephrased as follows.

\begin{conjecture}\cite[Conjecture 2. (i)]{ELX}\label{elx}
Let $w_i \ (1 \le i \le 2d+1)$ denote the $2d+1$ highest possible values of $|\V(F) \cap V_2|$, if $F \in \Fqt[x_0,x_1,x_2,x_3]$ is a homogeneous polynomial of degree $d$. Then for each $i$ there exist linear homogeneous polynomials $H_1,\dots,H_d$ such that
$V(H_1), \dots, V(H_d)$ contain a common line and $|\V(H_1\cdots H_d) \cap V_2|=w_i.$
\end{conjecture}

\begin{proposition}\label{false}
If $d=3$ and $q \ge 4$ then Conjecture \ref{elx} is false.
\end{proposition}

\begin{proof}
Let $q \ge 4$. By Theorem \ref{2nd} the second highest number of points that a cubic surface and $V_2$ have in common is given by $3(q^3 + q^2 - q) + 1$. Now, let $\Pi_1, \Pi_2, \Pi_3$ be three planes in $\PP^3(\Fqt)$ containing a common line $\ell$. Take $X = \Pi_1 \cup \Pi_2 \cup \Pi_3$. If one of the $\Pi_j$-s is not tangent to $V_2$, then from Lemma \ref{nontan} we see that $|X \cap V_2| < 3q^2 + 2q^2 +1$. Now suppose that $\Pi_1, \Pi_2, \Pi_3$ are all tangent to $V_2$. Then $\ell$ is either a secant line or a generator. If $\ell$ is a secant line, then by Proposition \ref{attained} we have $|X \cap V_2| = 3(q^3 + q^2 - q) + q + 1$. On the other hand, if $\ell$ is a generator, then $|X \cap V_2| = 3q^3 + q^2 + 1$. Thus such a configuration does not give rise to the desired number of points of intersections with $V_2$.
\end{proof}

\section{Acknowledgments}
The authors would like to acknowledge the support from The Danish Council for Independent Research (Grant No. DFF-8021-00030B and DFF-6108-00362 respectively). Also, the authors are grateful for the valuable suggestions and comments of the referee. Specifically, in the first version of this paper S\o rensen's conjecture was proved for $q \ge 8$, but the referee pointed out that using the theory of Eckardt points the conjecture could be settled for the remaining values of $q$ as well.


\begin{thebibliography}{GV1}
\bibitem{BC}
R.C. Bose,  I.M. Chakravarti,  Hermitian varieties in a finite projective space $PG(N,q^2)$. \emph{Canad. J. Math.} {\bf 18} (1966), 1161--1182.

\bibitem{C}
 I.M. Chakravarti,  Some properties and applications of Hermitian varieties in a finite projective space $PG(N,q^2)$ in the construction of strongly regular graphs (two-class association schemes) and block designs. \emph{J. Combinatorial Theory Ser. B} {\bf 11} (1971), 268--283.

\bibitem{Ch}
I.M. Chakravarti, The generalized Goppa codes and related discrete designs from Hermitian surfaces in PG($3,s^2$),  \emph{Lecture Notes in Comput. Sci.}  {\bf 311}, Springer, Berlin, 1986,  116--124.

\bibitem{CK}
A. Cossidente, G. Korchm{\'a}ros, Transitive ovoids of the Hermitian surface of PG($3,q^2$), with q even. \emph{J. Combin. Theory Ser. A.} {\bf 101} (2003), 117--130.

\bibitem{DG}
M. Datta, S.R. Ghorpade,  Number of solutions of systems of homogeneous polynomial equations over finite fields. \emph{Proc. Amer. Math. Soc}. {\bf 145} (2017), no. 2, 525--541.

\bibitem{DG1}
 M. Datta,  S.R. Ghorpade,  Remarks on the Tsfasman-Boguslavsky Conjecture and higher weights of projective Reed-Muller codes. Arithmetic, geometry, cryptography and coding theory, 157--169, \emph{Contemp. Math.,} {\bf 686}, Amer. Math. Soc., Providence, RI, 2017.

\bibitem{DD} I. Dolgachev, A. Duncan, Automorphisms of cubic surfaces in positive characteristic, arXiv:1712.01167, 2017 (https://arxiv.org/abs/1712.01167)

\bibitem{E}
 F.A.B. Edoukou, Codes defined by forms of degree 2 on Hermitian surfaces and S{\o}rensen's conjecture. \emph{Finite Fields Appl.} 13 (2007), no. 3, 616--627.

\bibitem{ELX}
F.A.B. Edoukou, S. Ling, C. Xing, Structure of functional codes defined on non-degenerate Hermitian varieties. \emph{J. Combin. Theory Ser. A} {\bf 118} (2011), no. 8, 2436--2444.

\bibitem{GL}
S.R. Ghorpade,  G. Lachaud,
{\'E}tale cohomology, Lefschetz theorems and number of points of singular varieties over finite fields.
\emph{Mosc. Math. J.} {\bf 2} (2002), no. 3, 589--631.

\bibitem{GK}
L. Giuzzi, G. Korchm{\'a}ros, Ovoids of the Hermitian Surface in Odd Characteristic, \emph{Adv. Geom.} \textbf{suppl.} (2003), S49--S58.

\bibitem{H}
J. Harris,  Algebraic geometry. A first course. Graduate Texts in Mathematics, 133. Springer-Verlag, New York, 1992.

\bibitem{Homma}
M. Homma,  A bound on the number of points of a curve in a projective space over a finite field. Theory and applications of finite fields, 103--110, \emph{Contemp. Math.} {\bf 579}, Amer. Math. Soc., Providence, RI, 2012.

\bibitem{IZ2}
S. Innamorati, M. Zannetti, F. Zuanni, Note A combinatorial characterization of the Hermitian surface, \emph{Discrete Math.} {\bf 313} (2013), 1496--1499.


\bibitem{L}
G. Lachaud,
Number of points of plane sections and linear codes defined on algebraic varieties. (English summary) \emph{Arithmetic, geometry and coding theory (Luminy, 1993)}, 77--104, de Gruyter, Berlin, 1996.


\bibitem{LR}
G. Lachaud, R. Rolland, On the number of points of algebraic sets over finite fields. \emph{J. Pure Appl. Algebra} 219 (2015), no. 11, 5117--5136.


\bibitem{R}
M. Reid, Undergraduate algebraic geometry. London Mathematical Society Student Texts, 12. Cambridge University Press, Cambridge, 1988.

\bibitem{S}
J.-P. Serre, Lettre \`a M. Tsfasman (French, with English summary), \emph{Journ{\'e}es Arithm{\'e}tiques}, 1989 (Luminy, 1989), Ast{\'e}risque 198--200 (1991), 11, 351--353 (1992).

\bibitem{So}
A.B.  S{\o}rensen, Projective Reed-Muller codes, \emph{IEEE Trans. Inform. Theory} {\bf 37} (1991), no. 6, 1567--1576.

\bibitem{SoT}
A.B. S{\o}rensen, Rational points on hypersurfaces, Reed-Muller codes and algebraic-geometric codes, Ph.D. thesis, Aarhus, Denmark, 1991.

\bibitem{ZS}
O. Zariski,  P. Samuel,  Commutative algebra. Vol. II. Reprint of the 1960 edition. Graduate Texts in Mathematics, Vol. 29. Springer-Verlag, New York-Heidelberg, 1975.
\end{thebibliography}
\end{document}